\documentclass[a4paper,12pt]{article}
\usepackage{amssymb}
\usepackage{latexsym}
\usepackage{amsmath,amsthm,amssymb}
\usepackage{amsfonts}
\usepackage{eufrak}

\newtheorem{thm}{Theorem}[section]
\newtheorem{lem}[thm]{Lemma}
\newtheorem{pro}[thm]{Proposition}

\theoremstyle{definition}

\newtheorem{exa}[thm]{Example}
\newtheorem{rem}[thm]{Remark}

\begin{document}

\begin{center}
{\Large Norms of sub-exponential random vectors}
\end{center}
\begin{center}
{\sc Krzysztof Zajkowski}\\
Institute of Mathematics, University of Bialystok \\ 
Ciolkowskiego 1M, 15-245 Bialystok, Poland \\ 
kryza@math.uwb.edu.pl 
\end{center}

\begin{abstract}
We discuss various forms of the Luxemburg norm in spaces of random vectors with coordinates belonging to the classical  Orlicz spaces of exponential type. We prove equivalent relations between some kinds of these forms. We also show when the so-called uniform norm is majorized by norms of coordinates up to some constants. We give an application of other norm to study of chaos in random vectors with sub-exponential coordinates. 


\end{abstract}


\section{Introduction}
The most important classes of exponential type Orlicz spaces of random variables are spaces generated by the functions $\psi_p(x)=\exp(|x|^p)-1$ ($p\ge 1$).  For $p=1$, we have the space of sub-exponential random variables and, for $p=2$, the space of sub-gaussian random variables. 

The Luxemburg norm for the function $\psi_p$ is called  $\psi_p$-norm of a random variable $\xi$ 
and has the form
$$
\|\xi\|_{\psi_p}=\inf \Big\{K>0:\; \mathbb{E}\exp(|\xi/K|^p)\le 2\Big\}.
$$
The Orlicz space $L_{\psi_p}=L_{\psi_p}(\Omega,\Sigma,\mathbb{P})$ consists of all random variables $\xi$ on the probability space $(\Omega,\Sigma,\mathbb{P})$ with finite Luxemburg norm, i.e.
$$
L_{\psi_p}:=\{\xi:\;\|\xi\|_{\psi_p}<\infty\}.
$$
Elements of $L_{\psi_p}$ we will call {\it $p$-sub-exponential random variables}. 

Let us note that $L_{\psi_{p_1}}\subset L_{\psi_{p_2}}$, if $p_1\ge p_2$, and moreover $L_\infty\subset L_{\psi_p} \subset L_r$ ($p,r\ge 1$), where $L_\infty,\;L_r$ denote the classical Lebesgue spaces. In other words,  spaces $L_{\psi_p}$ form increasing family, by decreasing $p$, smaller than all of 
$L_p$-spaces but larger than the space of bounded random variables $L_\infty$. 

For $p$-sub-exponential random variables one can formulate the following lemma, whose proof  can be found in \cite[Lem. 2.1]{Zaj1}
\begin{lem}
\label{charlem}
Let $\xi$ be a random variable  and $p\ge 1$. There exist positive constants $K,L,M$ 
such that 
the following conditions are equivalent:\\
 1. 
 $\mathbb{E}\exp(|\xi/K|^p)\le 2$ (i.e. $\xi$ is $p$-sub-exponential r.v. with $\|\xi\|_{\psi_p}\le K$);\\
 2. 
 $\mathbb{P}(|\xi|\ge t) \le 2\exp\big(-(t/L)^p\big)$ for all $t \ge 0$;\\ 
 3. 
 $\mathbb{E}|\xi|^\alpha\le  2M^\alpha\Gamma(\alpha/p+1)$ for all $\alpha\ge 1$.
\end{lem}
\begin{exa}
Apart from mentioned sub-exponential and sub-gaussian random variables, the Weibull random variables  with scale parameter $\lambda$ and shape parameter $p\ge 1$ form model examples of $p$-sub-exponential r.v.s, since their $\alpha$-moments equal $\lambda^\alpha\Gamma(\alpha/p+1)$. 
\end{exa}
\begin{rem}
Because for the gamma function we have $\Gamma(\alpha/p+1)^{1/\alpha}\sim_p\alpha^{1/p}$, where $\sim_p$ means that there is a constant $c_p$, depending on $p$, such that $c_p^{-1}\alpha^{1/p}\le \Gamma(\alpha/p+1)^{1/\alpha}\le c_p\alpha^{1/p}$, then, by property 3 of Lemma \ref{charlem}, one can introduce an equivalent definition  to the Luxemburg norm of the form $\sup_{\alpha\ge 1}\alpha^{-1/p}(\mathbb{E}|\xi|^\alpha)^{1/\alpha}$.
\end{rem}

Centered sub-gaussian random variables have another important classical characteristic (see Kahane \cite{Kahane}): a random variable $\xi$ 
is sub-gaussian if there exists a positive constant $K$ such that $\mathbb{E}\exp(t\xi)\le \exp(K^2t^2/2)$ for all $t\in\mathbb{R}$. In other words when the moment generating function of $\xi$ is majorized by the moment generating function of centered gaussian variable $g$ with the variation $K^2$.
\begin{rem}
It is possible to introduce the other  norm, equivalent to $\|\cdot\|_{\psi_2}$ in the space of centered sub-gaussian random variables, of the form
$$
\tau(\xi)=\inf\{K>0:\;\forall_{t\in\mathbb{R}}\;\mathbb{E}\exp(t\xi)\le \exp(K^2t^2/2)\};
$$
see Buldygin and Kozachenko \cite[Def.1.1.1]{BulKoz}. 
\end{rem}
 
Let us note that sums of independent, mean zero, sub-gaussian random variables $(\xi_i)_{i=1}^n$ have the very important {\it approximate rotation invariance} property:
$$
\tau^2\Big(\sum_{i=1}^n\xi_i\Big)\le \sum_{i=1}^n\tau^2(\xi_i);
$$
see Buldygin and Kozachenko \cite[Lem.1.1.7]{BulKoz}

In general, it is a problem to indicate some majorant for the moment generating function of a given $p$-sub-exponential random variable, for any $p\ge 1$, because this majorant must be not only $N$-function (see \cite[Def.2.2.2]{BulKoz}), as the function $|x|^p$ ($p>1$), but   must also satisfy so-called the quadratic condition, i.e. be quadratic function at a neighborhood of zero (see \cite[p.67]{BulKoz}). 

We propose some standardization $\varphi_p$ of the function $\phi_p(x):=|x|^p$, for $p>1$, in the form: $\varphi_p(x)=x^2/2$ if $|x|\le 1$ and  $\varphi_p(x)=|x|^p/p-1/p+1/2$ if $|x|>1$. The functions  $\varphi_p$ and $\phi_p$ are equivalent (\cite[Lem.2.5]{Zaj1}) and, for this reason, generate the same Orlicz spaces (see for instance \cite[Th. 2.3.2]{BulKoz}). Moreover, the Young transform of $\varphi_p$ equals $\varphi_q$ where $q=p/(p-1)$, what we write as $\varphi_p^\ast=\varphi_q$ (see \cite[Lem.2.6]{Zaj1}).
Similarly as in the sub-gaussian case, one can define $p$-sub-exponential norm of centered random variables in the form:
$$
\tau_{\varphi_p}(\xi)=\inf\{K>0:\;\forall_{t\in\mathbb{R}}\;\mathbb{E}\exp(t\xi)\le \exp\varphi_q(Kt)\}\quad(q=p/(p-1)).
$$
The general form of the above definition one can find in Buldygin and Kozachenko \cite[Def.2.4.1]{BulKoz}. 
It is known that this norm is equivalent to $\psi_p$-norm restricted to the space of centered random variables belonging to $L_{\psi_p}$ (compare Buldygin and Kozachenko \cite[Th.2.4.2 and Th.2.4.3]{BulKoz} and Giuliano Antonini et al. \cite[Cor.5.1]{Rita}, and  see especially \cite[Th.2.7]{Zaj1}). 

From now on let $c_p$ denote a constant such that
$c_p^{-1}\|\cdot\|_{\psi_p}\le \tau_{\varphi_p}(\cdot)\le c_p\|\cdot\|_{\psi_p}$ ($p>1$). Let us emphasize that, for $p=1$, there is not an equivalent of the norms $\tau_{\varphi_p}$, since the moment generating functions of $1$-sub-exponential (simply sub-exponential) random variables may be defined (take finite values) only on some neighborhoods of zero different for different  variables.

Let us stress that  sums of independent, mean zero, $p$-sub-exponential random variables $(\xi_i)_{i=1}^n$ possess the approximate rotational invariance property only for $p\ge 2$. In general, if $r=\min\{2,q\}$ then we have 
\begin{equation}
\label{rotinv}
\tau_{\varphi_p}^r\Big(\sum_{i=1}^n\xi_i\Big)\le \sum_{i=1}^n\tau_{\varphi_p}^r(\xi_i);
\end{equation}
compare Buldygin and Kozachenko \cite[Th.2.5.2]{BulKoz}.

\section{Norms of sub-exponential random vectors}
Let $\left\langle\cdot,\cdot\right\rangle$ stand for the standard inner product in $\mathbb{R}^n$. For $x=(x_i)_{i=1}^n\in\mathbb{R}^n$ and $p\ge 1$, let $|x|_p$ denote the $p$-th norm of $x$, i.e. $|x|_p=(\sum_{i=1}^n|x_i|^p)^{1/p}$ and $|x|_\infty=\max_{1\le i\le n}|x_i|$. 

We define three types of norms of random vectors $\xi=(\xi_i)_{i=1}^n$. 
First one is the simple maximum of $\psi_p$-norms of coordinates: $\max_{1\le i\le n}\|\xi_i\|_{\psi_p}$. 
For the following second norm we preserve the notation as in one dimension: 
$$
\|\xi\|_{\psi_p}:=\sup_{|{\bf t}|_q=1}\|\left\langle\xi,\bf{t}\right\rangle\|_{\psi_p}=\inf\Big\{K>0:\;\sup_{|{\bf t}|_q=1}\mathbb{E}\exp\Big(\Big|\frac{\left\langle\xi,\bf{t}\right\rangle}{K}\Big|^p\Big)\le 2\Big\}.
$$
The third norm we define as follows
$$
\|\xi\|_{E_p}:=\||\xi|_p\|_{\psi_p}=\inf\Big\{K>0:\;\mathbb{E}\exp\Big(\Big|\frac{\xi}{K}\Big|_p^p\Big)=\mathbb{E}\exp\Big(\frac{\sum_{i=1}^n|\xi_i|^p}{K^p}\Big)\le 2\Big\}.
$$

We have the following estimates of the Luxemburg norms of  random vector by the Luxemburg norms of its coordinates.
\begin{pro}
Let $\max_{1\le i\le n}\|\xi_i\|_{\psi_p}<\infty$ and $\xi=(\xi_1,...,\xi_n)$.
Then
$$
\max_{1\le i\le n}\|\xi_i\|_{\psi_p}\le\|\xi\|_{\psi_p}\le \|\xi\|_{E_p}\le n^{1/p}\max_{1\le i\le n}\|\xi_i\|_{\psi_p}.
$$
\end{pro}
\begin{proof}
Let $K\ge \|\xi\|_{\psi_p}$. Because $|{\bf e}_i|_p=1$ 
for all $p\in[1,\infty]$ (${\bf e}_i$ are vectors of the standard basis in $\mathbb{R}^n$) and $\left\langle\xi,{\bf e}_i\right\rangle=\xi_i$ then 
$$
\mathbb{E}\exp\Big(\Big|\frac{\xi_i}{K}\Big|^p\Big)\le \sup_{|{\bf t}|_q=1}\mathbb{E}\exp\Big(\Big|\frac{\left\langle\xi,\bf{t}\right\rangle}{K}\Big|^p\Big)\le 2.
$$
It means that $\|\xi_i\|_{\psi_p}\le\|\xi\|_{\psi_p}$ for all $1\le i \le n$. It implies the first inequality.

Let now $K\ge\|\xi\|_{E_p}$. By the H\"older inequality in $\mathbb{R}^n$ we have $\left\langle\xi,\bf{t}\right\rangle\le|\xi|_p|{\bf t}|_q=|\xi|_p$, if $|{\bf t}|_q=1$.  Thus
$$
\sup_{|{\bf t}|_q=1}\mathbb{E}\exp\Big(\Big|\frac{\left\langle\xi,\bf{t}\right\rangle}{K}\Big|^p\Big)\le \mathbb{E}\exp\Big(\Big|\frac{\xi}{K}\Big|_p^p\Big)\le 2.
$$ 
It gives the second inequality.

Assume now that $K\ge \max_{1\le i\le n}\|\xi_i\|_{\psi_p}$. 
Using the multi-factorial H\"older inequality with exponents $p_i=n$, $1\le i\le n$, we get
$$
\mathbb{E}\exp\Big(\frac{\sum_{i=1}^n|\xi_i|^p}{nK^p}\Big)=\mathbb{E}\Big(\prod_{i=1}^n\exp\Big(\frac{|\xi_i|^p}{nK^p}\Big)\Big)
\le\prod_{i=1}^n\Big(\mathbb{E}\exp\Big(\frac{|\xi_i|^p}{K^p}\Big)\Big)^{1/n}\le 2,
$$
since $K\ge \|\xi_i\|_{\psi_p}$, for $1\le i \le n$, and, in consequence, each factor is less or equal $2^{1/n}$. It implies the last inequality.
\end{proof}
By $L_{\psi_p}^n=L_{\psi_p}^n(\Omega,\Sigma,\mathbb{P})$ we will denote spaces of random vectors with $p$-sub-exponential coordinates. It can be expressed as
$$
L_{\psi_p}^n(\Omega,\Sigma,\mathbb{P})=\{\xi:\Omega\mapsto\mathbb{R}^n:\;\|\xi\|_{\psi_p}<\infty\}.
$$
Let us emphasize that in the above definition of the $L_{\psi_p}^n$-space we can use other equivalent norms but the norm $\|\cdot\|_{\psi_p}$ is the most convenient because for $p\ge 2$ and centered independent coordinates it is majorized by the maximum norms of coordinates up to constant that depends only on $p$ but not depend on dimension $n$. 
\begin{pro}
Let $p>1$ and $r=\min\{q,2\}$ ($q=p/(p-1)$). Assume that a random vector $\xi=(\xi_i)_{i=1}^n$ has independent, mean zero, $p$-sub-exponential coordinates. Then  
$$
\|\xi\|_{\psi_p}\le  n^{1/r-1/q}c_p^2\max_{1\le i\le n}\|\xi_i\|_{\psi_p}.
$$.
\end{pro}
\begin{rem}
If $p\ge 2$ then $q\le 2$ and $r=q$. Thus we get the estimate of $\psi_p$-norm of a random vector $\xi$ by maximum $\psi_p$-norms of its coordinates up to the constant that depends only on $p$. If $1< p <2$ then on the right hand side appears the additional factor $n^{1/2-1/q}$, which tends to $\sqrt{n}$ as $p\to 1$.
\end{rem}
\begin{proof}
Let ${\bf t}\in \mathbb{R}^n$ and $|{\bf t}|_q=1$. By equivalence of $\|\cdot\|_{\psi_p}$ and $\tau_p(\cdot)$ and  by the rotation invariance property (\ref{rotinv}) we get
\begin{eqnarray*}
\|\left\langle\xi,\bf{t}\right\rangle\|_{\psi_p}&\le& c_p\tau_{\varphi_p}\Big(\sum_{i=1}^nt_i\xi_i\Big)
\le c_p\Big(\sum_{i=1}^n\tau_{\varphi_p}^r(t_i\xi_i)\Big)^{1/r}\\ 
\; & \le & c_p\Big(\sum_{i=1}^nt_i^r\Big)^{1/r}\max_{1\le i\le n}\tau_{\varphi_p}(\xi_i)\le 
c_p^2|{\bf t}|_r\max_{1\le i\le n}\|\xi_i\|_{\psi_p}.
\end{eqnarray*}
Taking supremum over $|{\bf t}|_q=1$ we get our estimate on the $\psi_p$-norm of $\xi$:
$$
\|\xi\|_{\psi_p}=\sup_{|{\bf t}|_q=1}\|\left\langle\xi,\bf{t}\right\rangle\|_{\psi_p}\le(\sup_{|{\bf t}|_q=1}|{\bf t}|_r)c_p^2\max_{1\le i\le n}\|\xi_i\|_{\psi_p}
=n^{1/r-1/q}c_p^2\max_{1\le i\le n}\|\xi_i\|_{\psi_p},
$$
since $\sup_{|{\bf t}|_q=1}|{\bf t}|_r=n^{1/r-1/q}$ ($r\le q$).
\end{proof}

The $\psi_2$-norm is  best known, the most important and the most commonly used.  But sometimes it is natural to use other norms and not just for  $p=2$.
In the next section, we show an application of the norms $\|\cdot\|_{E_p}$ to study of chaos in sub-exponential random vectors.

\section{Chaos in sub-exponential random vectors}
To prove our main result of this section  we will need the  following technical lemma, whose proof can be found in \cite[Lem. 2.1]{Zaj2}.  
\begin{lem}
\label{lem2}
Let $\eta$ be a centered random variable. If there exist positive constants $a, b$   such that 
$\mathbb{E}\exp(t\eta)\le \exp(a^2t^2/2)$ if $t\in[-b,b]$ then 
for  every $s\ge 0$ we have
$$
\mathbb{P}\big(|\eta|\ge s\big)\le 2e^{-g(s)},
$$
where 
$$
g(s)=\left\{
\begin{array}{ccl}
\frac{s^2}{2a^2} & {\rm if} & 0\le s \leq a^2b,\\
bs-\frac{a^2b^2}{2} & {\rm if} &   a^2b<s.
\end{array}
\right.
$$
\end{lem}
\begin{rem}
\label{remlem2}
Let us observe that $g(t)\ge \min\{t^2/(2a^2),bt/2\}$ and we may rewrite the claim of the above lemma in a weaker but more traditional, for the Bernstein-type inequality, form  as follows 
$$
\mathbb{P}\big(|\eta|\ge t\big)\le 2\exp\Big(-\min\Big\{\frac{t^2}{2a^2},\frac{bt}{2}\Big\}\Big).
$$
\end{rem}
\begin{rem}
\label{rem2lem2}
Equivalence of properties 4 and 5 in Vershynin 
\cite[Lem. 2.1]{Ver} implies  that  for a centered sub-exponential random variable $\eta$ the assumptions of  Lemma \ref{lem2} are  satisfied with $a=\sqrt{2}C\|\eta\|_{\psi_1}$ and $b=1/(C\|\eta\|_{\psi_1})$, where $C$ is a universal constant. 
So, for such variables, we can rewrite the above estimate in the following form
$$
\mathbb{P}\big(|\eta|\ge t\big)\le 2\exp\Big(-\min\Big\{\frac{t^2}{4C^2\|\eta\|_{\psi_1}^2},\frac{t}{2C\|\eta\|_{\psi_1}}\Big\}\Big).
$$
\end{rem}

Let us emphasize that one of the ways to obtain Bernstein-type inequalities for sub-exponential random variables is to find (estimate) their $\psi_1$-norms. Now we proceed to an estimate of norms of chaos in sub-exponential random vectors but first we recall the notion of chaos of order $d$. 

Let $A$ be a  multi-indexed array of real numbers $[a_{i_1,\ldots,i_d}]_{i_1,...,i_d=1}^n$ and  
$$
S_d(x)=\sum_{i_1,...,i_d=1}^na_{i_1,\ldots,i_d}x_{i_1}\cdots x_{i_d}
$$
for $x=(x_1,...,x_n)\in\mathbb{R}^n$. For a random vector $\xi=(\xi_i)_{i=1}^n$ a random variable $S_d(\xi)$ is called {\it chaos of order $d$}.

Now we show some estimate of the $\psi_1$-norm of chaos of order $d$ in $d$-sub-exponential random vectors.
Let $\|A\|_p$ denote the $p$-th norm of $A$ for $p\in [1,\infty]$, i.e.
$$
\|A\|_p=\Big(\sum_{i_1,...,i_d=1}^n\big|a_{i_1,\ldots,i_d}\big|^p\Big)^{1/p} \;{\rm and}\quad\|A\|_\infty=\max_{1\le i_1,...,i_d\le n}\big|a_{i_1,\ldots,i_d}\big|.
$$  
Observe that applying the H\"older inequality to $S_d(x)$ with exponents  $p$ and $p/(p-1)$  we get 
\begin{eqnarray*}
\Big|\sum_{i_1,...,i_d=1}^na_{i_1,\ldots,i_d}x_{i_1}\cdots x_{i_d}\Big|&\le& \Big(\sum_{i_1,...,i_d=1}^n\big|a_{i_1,\ldots,i_d}\big|^{p/(p-1)}\Big)^{(p-1)/p}
\Big(\sum_{i_1,...,i_d=1}^n|x_{i_1}\cdots x_{i_d}|^p\Big)^{1/p}\\
\; &=& \|A\|_{\frac{p}{p-1}}\Big(\sum_{i=1}^n|x_i|^p\Big)^{d/p}=\|A\|_{\frac{p}{p-1}}|x|_p^d
\end{eqnarray*}
It means that 
\begin{equation}
\label{estS}
|S_d(x)|\le \|A\|_{\frac{p}{p-1}}|x|_p^d.
\end{equation}

Let us notice that for sub-exponential random vectors we have the following
\begin{lem}
\label{estnormSd}
Let $\xi\in  L^n_{\psi_d}$ then $S_d(\xi)\in L_{\psi_1}$ and
$$
\|S_d(\xi)\|_{\psi_1}\le \|A\|_{d^\prime}\|\xi\|_{E_d}^d,
$$
where $d^\prime=d/(d-1)$.
\end{lem}
\begin{proof}
By  (\ref{estS}) we get
$$
\mathbb{E}\exp\Big(\frac{|S_d(\xi)|}{\|A\|_{d^\prime}\|\xi\|^d_{E_d}}\Big)\le \mathbb{E}\exp\Big(\frac{|\xi|_d^d}{\|\xi\|^d_{E_d}}\Big)\le 2.
$$
It means that the Luxemburg norm of $S_d(\xi)$ in $L_{\psi_1}$ is less or equal $\|A\|_{d^\prime}\|\xi\|^d_{E_d}$.
\end{proof}

By the definition of the norm $\|\cdot\|_{\psi_1}$ and the Jensen inequality applied to a convex function $\exp\{|\cdot|/a\}$ ($a>0$) we get
\begin{equation*}
\label{estE}
2\ge \mathbb{E}\exp\Big(\frac{|S_d(\xi)|}{\|S_d(\xi)\|_{\psi_1}}\Big)\ge \exp\Big(\frac{|\mathbb{E}S_d(\xi)|}{\|S_d(\xi)\|_{\psi_1}}\Big)=
\mathbb{E}\exp\Big(\frac{|\mathbb{E}S_d(\xi)|}{\|S_d(\xi)\|_{\psi_1}}\Big),
\end{equation*}
which means that $\|\mathbb{E}S_d(\xi)\|_{\psi_1}\le\|S_d(\xi)\|_{\psi_1}$.

To sum up, we get the following estimate for $\psi_1$-norm of centered chaos of order $d$ in $d$-sub-exponential random vectors: 
$$
\|S_d(\xi)-\mathbb{E}S_d(\xi)\|_{\psi_1}\le 2\|S_d(\xi)\|_{\psi_1}\le 2\|A\|_{d^\prime}\|\xi\|_{E_d}^d.
$$

By Remark \ref{rem2lem2} and the above inequality, taking in Lemma \ref{lem2} 
$a=2\sqrt{2}C\|A\|_{d^\prime}\|\xi\|_{E_d}^d$ and $b=1/(2C\|A\|_{d^\prime}\|\xi\|_{E_d}^d)$ we get the main result of this section.
\begin{pro}
Let  $\xi\in L^n_{\psi_d}(\Omega)$ 
then
$$
\mathbb{P}\Big(\big|S_d(\xi)-\mathbb{E}S_d(\xi)\big|\ge t\Big)\le 2e^{-g(t)},
$$
where 
$$
g(t)=\left\{
\begin{array}{ccl}
\frac{t^2}{16C^2\|A\|_{d^\prime}^2\|\xi\|^{2d}_{E_d}} & {\rm if} & 0 \le t \le 4C\|A\|_{d^\prime}\|\xi\|^d_{E_d}\\
\frac{t}{2C\|A\|_{d^\prime}\|\xi\|^d_{E_d}}-1 & {\rm if} & t > 4C\|A\|_{d^\prime}\|\xi\|^d_{E_d}
\end{array}
\right.
$$
and $d^\prime=d/(d-1)$ and $C$ is the absolute constant as in Remark \ref{rem2lem2}.
\end{pro}
\begin{rem}
For $d=2$ we have quadratic forms $S_2(\xi)=\xi^TA\xi$ ($A=[a_{ij}]_{1\le i,j\le n}$) of sub-gaussian random vectors $\xi$ ($\|\xi\|_{E_2}<\infty$). Let $K$ denote the norm $\|\xi\|_{E_2}$. In the case of quadratic forms we can take the operator norm $\|A\|$ instead of the Hilbert-Schmidt norm $\|A\|_2$. Using the weaker form of Bernstein-type estimate (see Rem. \ref{remlem2}) we get 
$$
\mathbb{P}\Big(\big|\xi^TA\xi-\mathbb{E}(\xi^TA\xi)\big|\ge t\Big)\le 
2\exp\Big(-\min\Big\{\frac{t^2}{16C^2\|A\|^2K^4}, \frac{t}{4C\|A\|K^2} \Big\}\Big).
$$
The similar result was obtained in \cite[Th.2.5]{Ada} but under the assumption that $\xi$ has the convex concentration property with a constant $K$ (see therein).
Let us emphasize that in our approach $K$ is the Luxemburg norm $\|\xi\|_{E_2}$.  
\end{rem}

The above result  has the form of the Hanson-Wright inequality for independent random variables which was proved in \cite{HanWri,Wri} and recently derived in \cite{RudVer}. In \cite[Lem.2.5, Lem.2.7]{Zaj2} one can find another estimate of quadratic forms in dependent sub-gaussian random variables with different norms of random vectors 
and another form of bounds on tail probabilities (see \cite[Prop.2.6, Th.2.9]{Zaj2}). Let us emphasize that the introduced norm $\|\cdot\|_{E_2}$  can be used for investigation of chaos of order greater than two. 


\end{document}